\documentclass[12pt,psamsfonts,openany]{article}
\usepackage{a4,amsmath,amssymb,amsopn,amsthm,graphics,latexsym}
\newtheorem{theorem}{Theorem}
\newtheorem{lemma}{Lemma}
\newtheorem{proposition}{Proposition}
\newtheorem{definition}{Definition}
\newtheorem{remark}{Remark}
\newtheorem{property}{Property}

\newcommand\AC{{\rm{(AC)}}}
\newcommand\sign{{\rm{sign}}}
\newcommand\tu{{\widetilde{U}}}
\newcommand\me{\kern.3mm\mathbb{E}\kern.3mm}
\newcommand\dv{\kern.3mm{\rm{:}}\ }
\newcommand\tz{\kern.3mm{\rm{;}}\ }
\newcommand\sdots{.\kern.3mm.\kern.3mm.}

\newcommand {\beq}{\begin{equation}}
\newcommand {\eeq}{\end{equation}}
\newcommand {\barr}{\begin{eqnarray}}
\newcommand {\barn}{\begin{eqnarray*}}
\newcommand {\earr}{\end{eqnarray}}
\newcommand {\earn}{\end{eqnarray*}}
\newcommand{\bexam}{\begin{example}}
\newcommand{\eexam}{\end{example}}
\newcommand{\bpro}{\begin{property}}
\newcommand{\epro}{\end{property}}
\newcommand{\ble}{\begin{lemma}}
\newcommand{\ele}{\end{lemma}}
\newcommand{\bre}{\begin{remark}}
\newcommand{\ere}{\end{remark}}
\newcommand{\bdef}{\begin{definition}}
\newcommand{\edf}{\end{definition}}
\newcommand{\bth}{\begin{theorem}}
\newcommand{\enth}{\end{theorem}}
\newcommand{\bco}{\begin{corollary}}
\newcommand{\eco}{\end{corollary}}

\begin{document}

\begin{center}
{\bf \Large Orthogonal series
and limit theorems for canonical $U$- and $V$-statistics of
stationary connected observations}
\end{center}
%\toctitle{Orthogonal Series and Limit Theorems for Canonical
%$U$- and $V$-Statistics of
%Stationary Connected Observations}
%\titlerunning{ORTHOGONAL SERIES AND LIMIT THEOREMS}

%\Address{������� ����� ���������}
%         {�������� ����������\\
%          ��.~�.~�.~�������� �� ���,\\
%          ��.~��������� �������, 4,\\
%          ������������� ���������������\\
%          �����������,\\
%          �����������, 630090  ������.\\
%          E-mail: sibam@math.nsc.ru}
% \Address{�������� ������� ������������}
%         {�������� ����������\\
%          ��.~�.~�.~�������� �� ���,\\
%          ��.~��������� �������, 4,\\
%          ������������� ���������������\\
%          �����������,\\
%          �����������, 630090 ������.\\
%          E-mail: nvolodko@gmail.com}}
%

%\author{\firstname{I.~S.}~\surname{Borisov}}
%\email[E-mail: ]{sibam@math.nsc.ru}
%\affiliation{Sobolev Institute of Mathematics, Novosibirsk, 630090 Russia}

%\author{\firstname{N.~V.}~\surname{Volodko}}
%\email[E-mail: ]{nvolodko@gmail.com}
%\affiliation{Sobolev Institute of Mathematics, Novosibirsk, 630090 Russia}
%\tocauthor{I.~S.~Borisov and N.~V.~Volodko}
%\authorrunning{BORISOV, VOLODKO}

%\received{April 19, 2007}

\begin{abstract}
\noindent The limit behavior is studied for the distributions of
normalized~$U$- and $V$-statistics of an arbitrary order with
{\it canonical\/} (degenerate) kernels, based on samples of
increasing sizes from a stationary sequence of observations
satisfying~$\varphi$- or $\alpha$-mixing.  The corresponding
limit distributions are represented as infinite multilinear forms
of a centered Gaussian sequence with a known covariance matrix.

\indent{\it \bf keywords}: stationary sequence of random variables,
mixing, multiple orthogonal series, canonical $U$- and
$V$-statistics.
\end{abstract}

%\maketitle

\section{Introduction. Preliminary results}
%{\bf 1 Introduction. Preliminary results}

In~the present paper, we study the limit behavior of the
distributions  of normalized canonical~$U$- and $V$-statistics
based on stationary observations under~$\varphi$- or
$\alpha$-mixing. The approach based on a kernel representation of
the statistics under consideration as a multiple series, is quite
similar to the approach in~\cite{rubvit:BORVOL} where the
analogous results were obtained for independent observations.
First of all, introduce some definitions and notions
(see~\cite{Kolmogorov:BORVOL,KB:BORVOL}).

\begin{definition}
Let $\{\mathfrak X,\cal A\}$~be a measurable space
with a measure $\mu$. We say that the measure~$\mu$ has a~{\it countable basis} if there exists a countable family
%\vskip-3pt\noindent
$$
\mathfrak{A}:=\{A_n;\ n=1,2,\dots\}
$$
of measurable subsets (a countable basis of the measure~$\mu$) such that, for every
$M\in\cal{A}$ and any~$\varepsilon>0$, there is $A_{k}\in\mathfrak{A}$ such that
$$
\mu(M\Delta A_k)\le\varepsilon.
$$
\end{definition}

For example, if $\mathfrak X$~is a separable metric space and $\cal A$~is the  corresponding Borel $\sigma$-field then every $\sigma$-finite measure on~$\cal A$ has a countable basis.

Let $X_1,X_2,\dots$~be a stationary sequence of random variables
defined on a probability space $(\Omega,\cal F,\mathbb P)$ and
taking values in an  arbitrary measurable space~$\{\mathfrak
X,\cal A\}$. Denote by~$F$ the distribution of~$X_{1}$. In the
sequel, we consider only the distributions on~$\cal A$ which have
a countable basis. In addition to the stationary sequence
introduced above,  we need an auxiliary sequence~$\{X_i^{*}\}$
consisting of independent copies  of~$X_1$. Denote by
~$L_2(\mathfrak{X}^m,F^m)$ the space of measurable
functions~$f(t_1,\dots,t_m)$ defined on the corresponding
Cartesian power of the space~$\{\mathfrak X,\cal A\}$ with the
corresponding product-measure and satisfying the condition~$\me
f^2(X_1^*,\dots,X_m^*)<\infty$.

\begin{definition}                      %����������� 2
A function $f(t_1,\dots,t_m)\in L_2(\mathfrak{X}^m,F^m)$
is called  {\it canonical\/} (or {\it degenerate}) if
%$$
\begin{equation}\label{deg:BORVOL}                       %(1)
\me_{X^*_k} f(X_1^*,\dots,X_m^*)=0 \,\,\,a.\,s.
\end{equation}
%$$
%\vskip-7pt\noindent
for every $k$, where $\me_{X^*_k}$ is the conditional expectation given
the random variables $\{X^*_i; i\neq k\}$.
\end{definition}

Define a {\it canonical} Von Mises statistic by the formula
%\vskip-0.001pt\noindent
%$$
\begin{equation}\label{Mis:BORVOL}                         %(2)
V_n\equiv V_n(f):=n^{-m/2}
\sum_{1\leq j_1,\dots,j_m\leq n}f(X_{j_1},\dots,X_{j_m}),
\end{equation}
%$$
%\vskip-7pt\noindent 
where the function $f(t_1,\dots,t_m)$ (the
so-called \textit{kernel} of the statistic) is canonical. For
independent~$\{X_i\}$, such statistics called  {\it canonical
$V$-statistics} as well, are studied during last sixty years (see
the reference and examples of such
statistics~in~\cite{KB:BORVOL}). For the first time, some limit
theorems in the bivariate case were obtained in
\cite{Mises:BORVOL,Hf:BORVOL}. In addition to $V$-statistics, the
so-called $U$-statistics were studied as well:
%\vskip-0.001pt\noindent
%$$
\begin{equation}\label{Ustat:BORVOL}                 %(3)
U_n\equiv U_n(f):={n^{-m/2}}
\sum_{1\leq i_1\neq\cdots\,\neq i_m\leq n}f(X_{i_1},\dots,X_{i_m})
\end{equation}
%$$
%\vskip-0.1pt\noindent
or
%\vskip-7pt\noindent
%$$
\begin{equation}\label{Ustat0:BORVOL}                %(4)
U^0_n:={n^{-{m}/{2}}}
\sum_{1\leq i_1<\cdots\,<i_m\leq n}f_0(X_{i_1},\dots,X_{i_m}).
\end{equation}
%$$
Notice that, in addition, the kernel~$f_0$
in~(\ref{Ustat0:BORVOL}) is assumed to be symmetric (or invariant
with respect to all permutations of the variables).
%������� �����, ��� ���
%<<������������>> ����������� ��������� ���������~$V_n$, $U_n$ �~$U^0_n$
%�������  �������� $n^{-m/2}$ �~(\ref{Ustat0}) ��~$(C^m_n)^{-1/2}$.

The main difference of  $U$-statistics from $V$-statistics is
that, in the region of summation of the corresponding multiple
sums~(\ref{Ustat:BORVOL}) and~(\ref{Ustat0:BORVOL}), the so-called
{\it diagonal subspaces} are absent, i.~e., all the subscripts are
pairwise distinct. If the distribution~$F$ of the random
variable~$X_1$ has no atoms then setting the kernel to be zero on
all the diagonal subspaces, we can easily reduce $U$-statistics
in~(\ref{Ustat:BORVOL}) to $V$-statistics~(\ref{Mis:BORVOL}).
Moreover, it is easy to see that formulas~(\ref{Ustat:BORVOL})
and~(\ref{Ustat0:BORVOL}) for $U$-statistics are equivalent: If
in~(\ref{Ustat0:BORVOL}) we set $$ f_0(t_1,\dots,t_m):=\sum
f(t_{i_1},\dots,t_{i_m}), $$ %\vskip-7pt\noindent
where the sum is
taken over all permutations~$i_1,\dots,i_m$ of the
numbers~$1,\dots,m$, then we reduce the representation
in~(\ref{Ustat:BORVOL}) to that in~(\ref{Ustat0:BORVOL}).

Notice also that any $U$-statistic is represented as a finite
linear combination of canonical $U$-statistics of orders from~1
to~$m$ (called a {\it H{\"o}ffding decomposition\/},
see~\cite{KB:BORVOL}). This fact allows us to reduce an asymptotic
analysis of arbitrary $U$-statistics to that for canonical ones.

If the distribution~$F$ is arbitrary (in particular, it contains
atoms) that $U$-statistics in~(\ref{Ustat:BORVOL})
or~(\ref{Ustat0:BORVOL}) admit  recurrent representations as
linear combinations of canonical $V$-statistics and
$U$-statistics of smaller orders. In other words, we can
represent any canonical $U$-statistic as a linear combination of
canonical $V$-statistics. This is a key remark to study the limit
behavior of the distributions of $U$-statistics. In this
connection, note an important role of the $V$-statistics with
splitting kernels: %\vskip-1pt\noindent
%$$
\begin{equation}\label{product:BORVOL}                %(5)
f(t_1,\dots,t_m)=h_{1}(t_1)h_{2}(t_2)\cdots\,h_{m}(t_m)
\end{equation}
%$$
since, in this case, the corresponding $V$-statistic is represented  in the form
%\vskip-7pt\noindent
$$
V_n=\frac{1}{\sqrt{n}}
    \sum_{i=1}^nh_1(X_i)\cdots\,
    \frac{1}{\sqrt{n}}
    \sum_{i=1}^nh_m(X_i),
$$
where $\me h_k(X_1)=0$ and~$\me h^2_k(X_1)<\infty$.
So, under some dependency conditions of the random variables~$\{X_j\}$, one can apply the multivariate central limit theorem which describes the weak limit of  the $V$-statistic under consideration as the product
 $\prod_{j=1}^{m}\rho_j$, where $\{\rho_j;\ j\le m\}$~are centered Gaussian random variables with the covariance matrix
% \vskip-3pt\noindent
$$
\me\rho_{i}\rho_{j}=
\me h_i(X_1)h_j(X_1)
+\sum_{k=1}^{\infty}
\big(\me h_i(X_1)h_j(X_{k+1})+\me h_j(X_1)h_i(X_{k+1})
\big)
$$
%\vskip-4pt\noindent
provided that the series on the right-hand side of this equality exists.
 It is clear that these arguments are not essentially changed when the kernel of a $V$-statistic can be represented as a linear combination  (maybe infinite) of splitting kernels  (see the proof of Theorem~2). The same argument may be applied to $U$-statistics with such kernels due to the above-mentioned connection of~$U$- and $V$-statistics. This approach was applied to prove some well-known results on asymptotic analysis of
$U$-statistics of independent observations~$\{X_j\}$.

In this connection, recall some classic results connected with
expansion of a canonical function into a multiple orthogonal
series with respect to an orthogonal basis of the Hilbert
space~$L_2(\mathfrak{X},F)$. Since the distribution~$F$ has a
countable basis, the Hilbert space~$L_2(\mathfrak{X},F)$ is
separable. It means that, in this space, there exists a countable
orthonormal basis. Put~$e_0(t)\equiv 1$. Using the Gram~---
Schmidt orthogonalization~\cite{Kolmogorov:BORVOL}, one can
construct an orthonormal basis in~$L_2(\mathfrak{X},F)$ containing
the constant function $e_0(t)$. Denote
by~$\big\{e_i(t)\big\}_{i\geq 0}$ such basis. Then~$\me
e_i(X_1)=0$ for every~$i\geq 1$ due to the orthogonality of all
the other basis elements to the function~$e_0(t)$. The normalizing
condition means that~$\me e^2_i(X_1)=1$ for all~$i\geq 1$. Notice
that the collection of functions $$
\big\{e_{i_1}(t_1)e_{i_2}(t_2)\cdots\,e_{i_m}(t_m); \
i_1,i_2,\dots,i_m=0,1,\dots \big\} $$ is an orthonormal basis in
the Hilbert space $L_2(\mathfrak{X}^m,F^m)$ (for example,
see~\cite{Kolmogorov:BORVOL}).

Thus, one can represent the kernel $f(t_1,\dots,t_m)$ of the
statistics under consideration as a multiple orthogonal series
in~$L_2(\mathfrak{X}^m,F^m)$:
%$$
\begin{equation}\label{kernel1:BORVOL}               %(6)
f(t_1,\dots,t_m)=
\sum_{i_1,\dots,i_m=0}^{\infty}
f_{i_1,\dots,i_m}
e_{i_1}(t_1)\cdots\,e_{i_m}(t_m),
\end{equation}
%$$
where the series on the right-hand side of equality~(\ref{kernel1:BORVOL}) converges in the norm of~$L_2(\mathfrak{X}^m,F^m)$.
Moreover, if the coefficients~$\{f_{i_1,\dots,i_m}\}$
are absolutely summable then, due to the B. Levi theorem and the simple estimate
${\bf E}\big\vert e_{i_1}(X^*_1)\cdots\,e_{i_m}(X^*_m)\big\vert\le 1$,
the series in~(\ref{kernel1:BORVOL}) converges almost surely with respect to the distribution ~$F^m$ of the vector $(X^*_1,\dots,X^*_m)$.

Consider the case $m=2$ and the integral linear operator with a symmetric
kernel $f\in L_2(\mathfrak{X}^2,F^2)$ mapping the space~$L_2(\mathfrak{X},F)$ into itself. Since this linear operator is completely continuous and self-conjugate, in the separable Hilbert space~$L_2(\mathfrak{X},F)$, there exists an orthonormal basis
consisting of eigenvectors of this integral operator and, for this basis, representation~(\ref{kernel1:BORVOL}) for  $m=2$ is valid. Multiply by an arbitrary element~$e_k(t_2)$ the both sides of~(\ref{kernel1:BORVOL}) and integrate these modified parts with respect to the distribution~$F(dt_2)$. Taking orthogonality of the basis elements into account we obtain the new identity
$$
\lambda_ke_k(t_1)=\sum_{i=0}^{\infty}f_{i,k}e_{i}(t_1),
$$
where $\lambda_k$~is the corresponding eigenvalue. From here it immediately follows that~$f_{k,k}=\lambda_k$ and~$f_{i,k}=0$ for~$i\neq k$.
Therefore, for this basis in the case~$m=2$, formula (\ref{kernel1:BORVOL})
has the form
%$$
\begin{equation}\label{eagl:BORVOL}                   %(7)
f(t_1,t_2)=
\sum_{k=0}^{\infty}\lambda_ke_k(t_1)e_k(t_2),
\end{equation}
%$$
which was repeatedly employed by many authors.

Notice also the following property of canonical kernels.

\begin{proposition}                  %����������� 1
If $f(t_1,\dots,t_m)$~is a canonical kernel then $e_0(t)$ is absent in expansion~$(\ref{kernel1:BORVOL})$, i.~e., expansion~$(\ref{kernel1:BORVOL})$
has the form
%$$
\begin{equation}\label{kernel:BORVOL}                 %(8)
f(t_1,\dots,t_m)=
\sum_{i_1,\dots,i_m=1}^{\infty}f_{i_1,\dots,i_m}e_{i_1}(t_1)
\cdots\,e_{i_m}(t_m).
\end{equation}
\end{proposition}

\begin{proof}
%Verify degeneracy condition~(1). Since $f(t_1,\dots,t_m)\in
%L_2(\mathfrak{X}^m,F^m)$, by Fubini's theorem, the signs of expectation and %summation below can be permuted:
Denote
$$f_N(t_1,...,t_m)=\sum_{i_1, \dots, i_m=0}^{N}f_{i_1\dots i_m}e_{i_1}(t_1) \dots e_{i_m}(t_m).$$
It is clear that
$$\mathbb{E}_{X^*_k}f_N(X^*_1,...,X^*_m)=\sum_{i_1, \dots, i_m=0}^{N}f_{i_1,\dots,i_{k-1},0,i_{k+1},\dots,i_m}e_{i_1}(X^*_1) \dots e_{i_{k-1}}(X^*_{k-1})e_{i_{k+1}}(X^*_{k+1}),\dots,e_{i_m}(X^*_m).$$
%$$\mathbb{E}_{X^*_1}f(X^*_1,...,X^*_m)=0\,\,\,a.\,s.$$
Taking  degeneracy condition (\ref{deg:BORVOL}) into account, we obtain, as $N\to\infty$,
$$\mathbb{E}(\mathbb{E}_{X^*_k}f_N(X^*_1,...,X^*_m))^2=
\mathbb{E}\big(\mathbb{E}_{X^*_k}(f(X^*_1,...,X^*_m)-f_N(X^*_1,...,X^*_m))\big)^2$$
$$\leq\mathbb{E}\mathbb{E}_{X^*_k}\big(f(X^*_1,...,X^*_m)-f_N(X^*_1,...,X^*_m)\big)^2=
\mathbb{E}\big(f(X^*_1,...,X^*_m)-f_N(X^*_1,...,X^*_m)\big)^2\rightarrow 0.$$
In other words,  as $N\rightarrow\infty$,
\begin{align*}
\mathbb{E}\Big(\sum_{i_1, \dots, i_m=0}^{N}f_{i_1,\dots,i_{k-1},0,i_{k+1},\dots,i_m}e_{i_1}(X^*_1) \dots e_{i_1}(X^*_{k-1})e_{i_1}(X^*_{k+1}),\dots,e_{i_m}(X^*_m)\Big)^2\\
=\sum_{i_1, \dots, i_m=0}^{N}f_{i_1,\dots,i_{k-1},0,i_{k+1},\dots,i_m}^2\rightarrow 0
\end{align*}
due to the orthonormality of the basis functions. The last relation means that all the coefficients  $f_{i_1,\dots,i_{k-1},0,i_{k+1},\dots,i_m}$ vanish for all $k=1,\dots, m$ and all the subscripts $i_1,\dots,i_{k-1},i_{k+1},\dots,i_m$.
%
%
%%$$
%\begin{align*}
%&\me f(t_1,\dots,t_{k-1},X_k,t_{k+1},\dots,t_m)\\
%&=\sum_{i_1,\dots,i_m=0}^{\infty}
%f_{i_1,\dots,i_m}e_{i_1}(t_1)
%\cdots\,\me e_{i_k}(X_k)
%\cdots\,e_{i_m}(t_m)\\
%&=\kern-2.2mm\sum_{i_1,\dots,i_{k-1},i_{k+1},\dots,i_m=0}^{\infty}
%\kern-12mm
%f_{i_1,\dots,i_{k-1},0,i_{k+1},\dots,i_m}
%\cdot e_{i_1}(t_1)
%\cdots\,e_{i_{k-1}}(t_{k-1})e_{i_{k+1}}(t_{k+1})
%\cdots\,e_{i_m}(t_m)=0.
%\end{align*}
%%$$
%The last equality but one follows from the fact that~$e_0(t)\equiv 1$ and~$\me e_i(X_1)=0$ for all~$i\geq\nobreak 1$. The collection of variables $t_1,\dots,t_{k-1},t_{k+1},\dots,t_n$ is arbitrarily chosen. Hence all the coefficients
%$f_{i_1,\dots,i_{k-1},0,i_{k+1},\dots,i_m}$ are equal to~$0$ due to uniqueness
%of the expansion of zero function with respect to a basis of the space\break
%$L_2(\mathfrak{X}^{m-1},F^{m-1})$. Since~$k$ is arbitrary, $e_0$ is absent in the expansion of the kernel. %����������� ��������.
\end{proof}

Notice also that if the kernel~$f(t_1,t_2)$ in~(\ref{eagl:BORVOL})
is canonical then the constant function~$e_0(t)$ is the eigenfunction corresponding to the eigenvalue~$\lambda_0=0$ of the integral operator. So, in this case, the summation in~(\ref{eagl:BORVOL}) starts with~$k=1$.

Thus, after replacement of the vector~$(t_1,\dots,t_m)$
by the independent observations~$(X^*_1,\dots,X^*_m)$, the partial sums of the series on the right-hand side of~(\ref{kernel:BORVOL}) \big(or of~(\ref{eagl:BORVOL}) in the case~$m=2$\big) mean-square converge to the random variable
$f(X^*_1,\dots,X^*_m)$ and hence they converge in distribution.
However, in the present paper, we deal with dependent observations
for which  this property in general is not valid.

                                              %�������� 2
\section{The main results for weakly dependent observations}

In the present paper, we study stationary sequences~$\{X_j\}$ satisfying certain mixing conditions. Recall the definitions of the most popular mixing conditions.
For~$j\le k$, denote by~$\mathfrak{M}^k_{j}$
the $\sigma$-field of all events generated by the random variables
$X_j,\dots,X_k$.

\setcounter{definition}{2}
\begin{definition}                    %����������� 3
A  sequence  $X_1,X_2,\dots$ satisfies
 $\alpha$-{\it mixing\/} (or {\it strong mixing}) if
$$
\alpha(i):=
\sup_{k\ge 1}\,
\sup_{A\,\in\,\mathfrak{M}^k_{1},\,
      B\,\in\,\mathfrak{M}^{\infty}_{k+i}}
\big\vert\mathbb{P}(AB)-\mathbb{P}(A)\mathbb{P}(B)
\big\vert\rightarrow 0\quad\mbox{as}\quad  i\rightarrow\infty.
$$
\end{definition}

\begin{definition}                    %����������� 4
A  sequence $X_1,\ X_2,\dots$ satisfies
$\varphi$-{\it mixing\/} (or
{\it uniformly strong mixing})
if
$$
\varphi(i):=
\sup_{k\ge 1}\,
\sup_{A\,\in\,\mathfrak{M}^k_{1},\,
      B\,\in\,\mathfrak{M}^{\infty}_{k+i},\,
              \mathbb{P}(A)>0}
\frac{\big\vert\mathbb{P}(AB)-
               \mathbb{P}(A)
               \mathbb{P}(B)
      \big\vert}
     {\mathbb{P}(A)}
\rightarrow 0\quad\mbox{as}\quad  i\rightarrow\infty.
$$
\end{definition}

\begin{definition}                    %����������� 5
A  sequence $X_1,X_2,\dots$ satisfies
$\psi$-{\it mixing} if
%\vskip-1pt\noindent
$$
\psi(i):=\sup_{k\ge 1}\,
         \sup_{A\,\in\,\mathfrak{M}^k_{1},\,
               B\,\in\,\mathfrak{M}^{\infty}_{k+i},\,
                \mathbb{P}(A)
                \mathbb{P}(B)>0}
\frac{\big\vert\mathbb{P}(AB)-
               \mathbb{P}(A)
               \mathbb{P}(B)
      \big\vert}
     {\mathbb{P}(A)\mathbb{P}(B)}
\rightarrow 0\quad\mbox{as}\quad  i\rightarrow\infty.
$$
\end{definition}

It is clear that the sequences~$\big\{\alpha(i)\big\}$,
$\big\{\varphi(i)\big\}$, and~$\big\{\psi(i)\big\}$ are nondecreasing and
$\psi$-mixing is stronger than
$\varphi$-mixing which in turn implies $\alpha$-mixing.

In the sequel, in the case of $\varphi$-mixing, we assume that
%\vskip-7pt\noindent
%$$
\begin{equation}\label{fi:BORVOL}                      %(9)
\sum_{k=1}^{\infty}\varphi^{1/2}(k)<\infty.
\end{equation}
%$$
Note that this known condition provides the cental limit theorem
for the corresponding stationary sequences
(for example, see~\cite{Billingsley:BORVOL}).

Introduce also the following restriction on finite-dimensional distributions
of the stationary sequence~$\{X_i\}$.

{\bf (AC)} {\sl For every collection of pairwise distinct subscripts
$(j_1,\dots,j_m)$, the distribution of $(X_{j_1},\dots,X_{j_m})$ is absolutely continuous with respect to the distribution of $(X^*_1,\dots,X^*_m)$}.

Notice that this restriction will be nontrivial only for sequences
under~$\alpha$- or $\varphi$-mixing because, in the case of $\psi$-mixing,
by induction on~$m$,
from Definition~5 we can easily deduce the inequality
$$
P(X_{j_1}\in A_1,\dots,X_{j_m}\in A_m)
\leq\big(1+\psi(1)
    \big)^m\prod_{k=1}^mP(X_{k}\in A_k)
$$
%\vskip-7pt\noindent
which is valid for every collection of Borel subsets
$(A_1,\sdots,A_m)$ and for every pairwise distinct subscripts
$(j_1,\dots,j_m)$.  From here condition~$\AC$ immediately follows.

\begin{remark}                            %��������� 1
As was mention before, the condition
%\vskip-1pt\noindent
%$$
\begin{equation}\label{abs:BORVOL}                  %(10)
\sum_{i_1,\dots,i_m=1}^{\infty}|f_{i_1,\dots,i_m}|<\infty
\end{equation}
%$$
%\vskip-4pt\noindent
implies convergence of the series in~(\ref{kernel:BORVOL}) almost surely
with respect to the distribution of the vector~$(X^*_{1},\dots,X^*_{m})$. So, under condition~$\AC$, this convergence is valid almost surely with respect to the distribution of the random vector~$(X_{j_1},\dots,X_{j_m})$. In other words,
if condition~$\AC$ is fulfilled then, for any pairwise distinct subscripts ${j_1},\dots,{j_m}$, we can substitute the random variables $X_{j_1},\dots,X_{j_m}$ for $t_1,\dots,t_m$ in~(\ref{kernel:BORVOL}).
\end{remark}

\begin{remark}                           %��������� 2
Under restriction~(\ref{abs:BORVOL}), one can sometimes obtain the above-mentioned   multiple series expansions without any restrictions like~$\AC$ on joint distributions of the initial stationary sequence. For example, if the kernel
$f(t_1,\dots,t_m)$~is continuous in~$\mathbb R^m$  and all the basis
functions $e_k(t)$ are continuous and bounded uniformly in~$k$ then, under condition~(\ref{abs:BORVOL}), equality~(\ref{kernel:BORVOL}) is transformed into the {\it
identity\/} in~$\mathbb R^m$ (see the proof of Theorem~2 below). Therefore, in this identity, one can substitute {\it arbitrarily dependent\/} random variables $X_{j_1},\dots,X_{j_m}$ (in particular, for coinciding subscripts
$j_k$) for the arguments
$t_1,\dots,t_m$.
\end{remark}

As an example of expansion~(\ref{kernel:BORVOL}) which is everywhere valid, consider the following symmetric kernel canonical with respect to the $[-1,1]$-uniform distribution:
$$
f^*(t,s):=\sign(ts)\min\big\{|t|,|s|
                       \big\}.
$$
This is the covariance function of the Gaussian process~$\sign(t)W\big(|t|\big)$ defined on~$[-1,1]$, where $W(t)$~is a standard Wiener process on the positive half-line. The eigenvalues  of the corresponding covariance operator, i.~e., of the integral operator with the kernel~$f^*(t,s)$ in the Hilbert space~$L_2[-1,1]$, are calculated by the same formula as that for the eigenvalues of the covariance operator of a standard Wiener process:
 $\lambda_k=\big(\pi(k-1/2)\big)^{-2}$,
$k=1,2,\dots$ (for example, see~\cite[�.~8.6]{GS:BORVOL}).
The corresponding eigenfunctions in~(\ref{kernel:BORVOL}) which form an orthonormal basis (together with the constant function $e_0(t)\equiv 1$)
in~$L_2[-1,1]$ are calculated by the formula $e_k(t)=\sqrt{2}\sin\big(\pi(k-1/2)t\big)$,
$k=1,2,\dots$.

It is easy to see that this example satisfies the conditions of Remark~2.
So, for arbitrarily dependent random variables~$X_1$ and~$X_2$ with the~$[-1,1]$-uniform distribution of each~$X_i$, the random variable~$f^*(X_1,X_2)$ can be represented in such a way:
%$$
\begin{equation}\label{min:BORVOL}                      %(11)
f^*(X_1,X_2)=
\sum_{k=1}^{\infty}
\frac{2\sin\big(\pi(k-1/2)X_1
           \big)
       \sin\big(\pi(k-1/2)X_2
           \big)}
     {\pi^2(k-1/2)^2},
\end{equation}
%$$
%\vskip-7pt\noindent
and the series absolutely converges  {\it everywhere}.

Under $\varphi$-mixing for the sequence~$\{X_i\}$
but without the restrictions mentioned in Remarks~1 and~2,
in general one cannot use expansions~(\ref{eagl:BORVOL})
or~(\ref{kernel:BORVOL}).
Such mistake is contained in~\cite{eaglson:BORVOL} (see also~\cite{KB:BORVOL}),
where it is claimed that, in the case of $\varphi$-mixing stationary observations for $m=2$, under condition~(\ref{fi:BORVOL}) only but without any restrictions like~$\AC$ and the regularity condition mentioned in Remark~2, the following assertion is valid:
%\vskip-0.0001pt\noindent
%$$
\begin{equation} \label{eagleresult:BORVOL}            %(12)
U_n\stackrel{d}\rightarrow
\sum_{k=1}^{\infty}\lambda_k(\tau_k^2-1),
\end{equation}
%$$
%\vskip-4pt\noindent
where $\{\lambda_k\}$~are the eigenvalues of the integral operator with the symmetric kernel~$f(t_1,t_2)$, which are assumed to be summable
\big(i.~e., under condition~(10)\big), and $\{\tau_k\}$~is a Gaussian sequence of centered random variables with the covariances
%\vskip-0.0001pt\noindent
%$$
\begin{equation}\label{eagleresult2:BORVOL}             %(13)
\me\tau_k\tau_l=
\me e_k(X_1)e_l(X_{1})+
\sum_{j=1}^{\infty}\big[\me e_k(X_1)e_l(X_{j+1})+
                        \me e_l(X_1)e_k(X_{j+1})
                   \big],
\end{equation}
%$$
%\vskip-4pt\noindent
where $\big\{e_k(t)\big\}$~are the eigenfunctions corresponding to the eigenvalues~$\{\lambda_k\}$ and forming  an orthonormal basis
in~$L_2(\mathfrak{X},F)$ \big(actually, the first summand on the right-hand side~(\ref{eagleresult2:BORVOL}) is the Kronecker symbol $\delta_{k,l}$\big).
   In~(12) and in the sequel, we admit degeneracy of some random variables~$\tau_i$. In other words, we add to the class of Gaussian distributions the all weak limits when the variance tends to zero.
Actually, the similar agreement is contained in~\cite{Billingsley:BORVOL}.
To prove relation~(12) in~\cite{eaglson:BORVOL}
the author use the expansion of the kernel~$f(t_1,t_2)$ in series~(\ref{eagl:BORVOL}) with respect to the basis~$\big\{e_k(t)\big\}$. Due to the above-mentioned arguments, we could now substitute the independent observations~$X^*_i$ and~$X^*_j$ for the variables~$t_1$ and~$t_2$. The same is true for a pair~$X_i$
and~$X_j$ from a stationary sequence satisfying condition~$\AC$.
However, in~\cite{eaglson:BORVOL}, the author substituted a pair~$X_i$
and~$X_j$ from an arbitrary stationary sequence under $\varphi$-mixing with restriction~(\ref{fi:BORVOL}) only. But under this replacement, the above-mentioned equalities may be not fulfilled with a nonzero probability. Moreover,
in this case,  the limit law in~(\ref{eagleresult:BORVOL}) may change the form.
The idea of constructing examples of such a kind is very simple: We need to construct a stationary sequence~$\{X_i\}$ with a non-atomic marginal  distribution, such that its elements~$X_i$ and~$X_j$ coincide with nonzero probability for some subscripts~$i\neq j$.  We then can change the values of ~$f$ on diagonal subspaces to break the above-mentioned identities with a nonzero probability when we replace the arguments~$X^*_i$
and~$X^*_j$ of the kernel with the dependent pair~$X_i$ and~$X_j$. The corresponding construction is contained in the proof of the following assertion.

\setcounter{proposition}{1}
\begin{proposition}                        %����������� 2
There exist a stationary sequence\,$\{\hskip-1pt X_i\}$ and a canonical kernel~$f(t_1,t_2)$ satisfying all the restrictions in~{\rm{\cite{eaglson:BORVOL}}}. However, under substituting the observations~$X_1$ and~$X_2$ for~$t_1$ and~$t_2$ respectively, the series in~$(\ref{eagl:BORVOL})$ does not coincide with the kernel. Moreover, the weak limit for the distributions of the $U$-statistics differs  from~$(\ref{eagleresult:BORVOL})$.
\end{proposition}

So, under certain conditions \big(say, conditions~(\ref{abs:BORVOL}) and~$\AC$\big), $U$-statistic~(\ref{Ustat:BORVOL})
can be represented as the following multiple series converging almost surely:
$$
U_n=n^{-m/2}
\sum_{i_1,\dots,i_m=1}^{\infty}f_{i_1,\dots,i_m}
\sum_{1\leq j_1\neq\dots\neq j_m\leq n}
e_{i_1}(X_{j_1})\cdots\,e_{i_m}(X_{j_m}).
$$

Further analysis is similar to that in the i.i.d. case, i.~e., it is reduced  to extraction of $V$-statistics with splitting kernels from the multiple sum on the right-hand side of this identity. The main fragment of the proof in~\cite{rubvit:BORVOL} is as follows: The value
%\vskip-0.0001pt\noindent
$$
U_n(e_{i_1}\cdots\,e_{i_m})=n^{-m/2}
\sum_{1\leq j_1\neq\dots\neq j_m\leq n}
e_{i_1}(X_{j_1})\cdots\,e_{i_m}(X_{j_m})
$$
%\vskip-3pt\noindent
is represented as a linear combination of products of the values
%\vskip-0.0001pt\noindent
$$
\frac{1}{\sqrt{n}}
\sum_{j=1}^ne_i(X_j),\,
\frac{1}{n}
\sum_{j=1}^ne_{i_1}(X_j)e_{i_2}(X_j),\,\dots,\,
\frac{1}{n^{k/2}}
\sum_{j=1}^ne_{i_1}(X_j)\cdots\,e_{i_k}(X_j).
$$
The proof has a combinatorial character and does not depend of joint distributions of the random variables~$\{X_j\}$. Further we apply the corresponding laws of large numbers as well as the central limit theorem and the following simple assertion.

\begin{proposition}                  %����������� 3
Let~$\Phi(x,y)$, $x\in{\mathbb R}^k$, $y\in{\mathbb R}^r$,~be a continuous function.
%������
%$$
%\sup_{|x|\,\le\,M}\big\vert\Phi(x,y)-\Phi(x,y_0)
%                  \big\vert\to 0,\quad y\to y_0,
%$$
%��� ������~$M>0$.
Let $\{\zeta_n\}$~be an arbitrary sequence of random vectors in~${\mathbb R}^k$ weakly converging to some random vector~$\zeta$. Let $\{\eta_n\}$~be a sequence of random vectors in~${\mathbb R}^r$ defined on a common probability space with~$\{\zeta_n\}$, which converge in probability to a constant vector~$c_0$.
Then the following weak convergence is valid:
$$
\Phi(\zeta_n, \eta_n) \stackrel{d} \rightarrow \Phi(\zeta,c_0).
$$
\end{proposition}

  In the present paper, $\Phi (x,y)$~is a polinomial of components of  vectors~$x$ and~$y$, and the sequence~$\zeta_n$ is defined by the formula
$$
\zeta_n:=\Bigg\{n^{-1/2}\sum_{j=1}^{n}e_{1}(X_j),\dots,n^{-1/2}
                       \sum_{j=1}^{n}e_{N}(X_j)
        \Bigg\},
$$
%\vskip-3pt\noindent
and $\eta_n$ is the finite collection
$$
\Bigg\{n^{-k/2}
      \sum_{j=1}^ne_{i_1}(X_j)\cdots\,e_{i_k}(X_j);
      \ 2\le k\le m,
      \ i_1,\dots,i_k\le N
\Bigg\};
$$
%\vskip-3pt\noindent
here $N$~is an arbitrary natural number.

So, under the condition
$f(t_1,\dots,t_m)\in L_2(\mathfrak{X}^m,F^m)$, in the case of
i.i.d. random variables~$\{X_i\}$, it was proved in~\cite{rubvit:BORVOL} that
%\vskip-0.0001pt\noindent
%$$
\begin{equation}\label{Ulim:BORVOL}                   %(14)
U_n\stackrel{d}\rightarrow
\sum_{i_1,\dots,i_m=1}^{\infty}f_{i_1,\dots,i_m}
\prod_{j=1}^{\infty}H_{\nu_j(i_1,\dots,i_m)}(\tau_j),
\end{equation}
%$$
%\vskip-3pt\noindent
where $\{\tau_i\}$~is a sequence of independent random variables having the standard Gaussian distribution, $\nu_j(i_1,\dots,i_m):=\sum_{r=1}^{m}\delta_{j,i_r}$,
%~is the number of subscripts  among the collection~$i_1,\dots,i_m$ coinciding %with the natural number~$j$,
and
$H_k(x)$~are Hermite polynomials defined by the formula
$$
H_k(x):=(-1)^ke^{x^2/2}
\frac{d^k}{dx^k}
\big(e^{-x^2/2}
\big)
$$
or by the recurrent formula
%$$
\begin{gather*}
H_0(x)\equiv 1,\ \ H_1(x)=x,\\
H_{n+1}(x)=xH_n(x)-nH_{n-1}(x).
\end{gather*}
%$$

Thus, $H_k(x)$ is a polynomial of degree~$k$ and the product on the right-hand side of~(\ref{Ulim:BORVOL}) can be represented in the form
%$$
\begin{equation}\label{represent:BORVOL}           %(15)
\prod_{j=\min\{i_k\}}^{\max\{i_k\}}
H_{\nu_j(i_1,\dots,i_m)}(\tau_j)=
H_{r_1}(\tau_{j_1})\cdots\,H_{r_s}(\tau_{j_s}),
\end{equation}
%$$
where the natural numbers $r_s$ and $j_s$ are defined by the relation
$r_s=\sum_{r=1}^{m}\delta_{j_s,i_r}$,
%~is the number of subscripts among $i_1,\dots,i_m$ coinciding with the %natural number~$j_l$,
and at that,
$\sum_{l\le s}r_l=m$ and $\min\{i_k\}\le j_l\le\max\{i_k\}$ for all~$l\le s$. Therefore, the right-hand of~(\ref{represent:BORVOL}) is a polynomial of degree~$m$ of the variables $\tau_{j_1},\dots,\tau_{j_s}$ and with coefficients having a universal upper bound depending on~$m$ only.

The goal of the present paper is to prove limit representations of the form~(\ref{Ulim:BORVOL}) in the case of weakly dependent random variables~$\{X_i\}$.

Introduce some additional restrictions on the mixing coefficients and the basis functions in the expansion in~(\ref{Ulim:BORVOL}). In the sequel, we assume that the stationary sequence~$\{X_i\}$ satisfies either
$\alpha$-mixing or $\varphi$-mixing, and moreover, the orthonormal basis~$\{e_i(t)\}$ in~$L_2(\mathfrak{X},F)$ with the original element  $e_0(t)\,{\equiv}\,\nobreak  1$,
satisfies the following additional restrictions:
\begin{enumerate}%[\quad\upshape (a)]
\item
In the case of $\varphi$-mixing, we assume condition (\ref{fi:BORVOL}) to be satisfied and
%$$
\begin{equation}\label{fibase:BORVOL}           %(16)
\sup_i\me\big\vert e_i(X_1)
         \big\vert^m<\infty;
\end{equation}
%$$
\item
In the case of $\alpha$-mixing, we assume that,
for some~$\varepsilon>\nobreak0$
and an even number~$c\geq m$,
%$$
\begin{align}
\label{albase:BORVOL}
%\begin{equation}\label{albase}            %(17)
\sup_i\me\big\vert e_i(X_1)
         \big\vert^{m+\varepsilon}
&<\infty,\\
%\end{equation}
%$$
%$$
%\begin{equation}
\label{alpha:BORVOL}             %(18)
\sum_{k=1}^{\infty}k^{c-2}
\alpha^{\varepsilon/(c+\varepsilon)}(k)
&<\infty.
%\end{equation}
\end{align}
%$$
\end{enumerate}

Further, introduce a sequence of Gaussian centered random variables~$\{\tau_i\}$ with the covariances
%$$
\begin{equation}\label{cov:BORVOL}                %(19)
\me\tau_k\tau_l=
\me e_k(X_1)e_l(X_{1})+
\sum_{j=1}^{\infty}
\big[\me e_k(X_1)e_l(X_{j+1})+
     \me e_l(X_1)e_k(X_{j+1})
\big].
\end{equation}
%$$
It is easy to see that, due to orthonormality of the basis $\{e_k\}$,
we have
%$$
\begin{gather*}
\me\tau^2_k=1+2\sum_{j=1}^{\infty}\me e_k(X_1)e_k(X_{j+1}),\\
\me\tau_k\tau_l=\sum_{j=1}^{\infty}
\big[\me e_k(X_1)e_l(X_{j+1})+
     \me e_l(X_1)e_k(X_{j+1})
\big],\,\,\,\,\,\,\,k\neq l.
\end{gather*}
%$$
   The existence of the series in~(\ref{cov:BORVOL}) follows from the above-mentioned restrictions on the mixing coefficients, and it will be proved later. In what follows, the Gaussian sequence~$\{\tau_i\}$ will play a role of the weak limit as~$n\to\nobreak\infty$ for the sequence

$$
\Bigg\{n^{-1/2}\sum_{j=1}^ne_1(X_j),
\ n^{-1/2}\sum_{j=1}^ne_2(X_j),\dots
\Bigg\}.
$$

The main results of the present paper are contained in the following two theorems.

\begin{theorem}                          %������� 1
Let one of the following two conditions be fulfilled\dv
\begin{enumerate}%[\quad\upshape (a)\!]
\item The stationary sequence~$\{X_i\}$ satisfies
$\varphi$-mixing, $(\ref{fi:BORVOL})$, and~$(\ref{fibase:BORVOL})$\tz
\item The stationary sequence~$\{X_i\}$ satisfies
$\alpha$-mixing, $(\ref{albase:BORVOL})$, and~$(\ref{alpha:BORVOL})$.
\end{enumerate}

Then, for any canonical kernel
$f(t_1,\dots,t_m)\in L_2(\mathfrak{X}^m,F^m)$, under conditions~$(\ref{abs:BORVOL})$
and~$\AC$, the following assertion holds\dv
%$$
\begin{equation}\label{main:BORVOL}                  %(20)
U_{n}(f)\stackrel{d}\rightarrow
\sum_{i_1,\dots,i_m=1}^{\infty}f_{i_1,\dots,i_m}
\prod_{j=1}^{\infty}H_{\nu_j(i_1,\dots,i_m)}(\tau_j),
\end{equation}
%$$
where $U_n(f)$~is a statistic of the form~$(\ref{Ustat:BORVOL})$ and the centered Gaussian sequence~$\{\tau_i\}$ has the covariance matrix defined in~$(\ref{cov:BORVOL})$.
\end{theorem}

\begin{theorem}                              %������� 2
Let $\mathfrak X$~be a separable metric space and let a canonical kernel~$f(t_1,\dots,t_m)$ be continuous (in every argument)
everywhere on~$\mathfrak X^m$, and let it satisfy~$(\ref{abs:BORVOL})$.
Moreover, if all the basis functions~$e_k(t)$ in~$(\ref{kernel:BORVOL})$
are continuous and uniformly bounded in~$t$ and~$k$, and one of the two conditions of Theorem~$1$ is valid then, as $n\to\infty$,
%$$
\begin{equation}\label{iid:BORVOL}             %(21)
V_n(f)\stackrel{d}\rightarrow
\sum_{i_1,\dots,i_m=1}^{\infty}
f_{i_1,\dots,i_m}\tau_{i_1}\cdots\,\tau_{i_m},
\end{equation}
%$$
where the Gaussian sequence~$\{\tau_i\}$ is defined in Theorem~$1$.
\end{theorem}

\begin{remark}                          %��������� 3.
It is known that, in the i.i.d. case,  condition~(\ref{abs:BORVOL}) of absolutely summability of the coefficients in the series expansion of the kernel can be weakened up to summability of the coefficients squared~\cite{rubvit:BORVOL,KB:BORVOL}.
In the same time, in limit theorems for the corresponding $V$-statistics, the latter condition does not describe the limit behavior since to define the weak limit, we need the existence of moments of the kernel  on all the diagonal subspaces.  For example, under the regularity conditions only (without (10))
of Theorem~2 for bivariate  $V$-statistics, the assumption of finiteness of
${\mathbb E}\big\vert f(X_1,X_1)\big\vert$ is equivalent to summability of the sequence~$\lambda_k\equiv f_{k,k}$ in representation~(\ref{eagl:BORVOL}), say, if all $\lambda_k$ are positive.
However, in the i.i.d. case, for the kernels of a bigger order, we need no summability of the coefficients~$f_{i_1,\dots,i_m}$ on the set of all pairwise distinct subscripts.
\end{remark}

As is noted in Proposition~4 below, in the case of dependent observations,
we cannot omit the above-mentioned restriction regarding summability of the coefficients
$f_{i_1,\dots,i_m}$ on the diagonal subspaces  for $U$-statistics as well.

\begin{proposition}                    %����������� 4.
There exist a stationary $1$-dependent sequence~$\{X_i\}$
satisfying condition~$\AC$, and a canonical kernel
$f(t_1,t_2)\in L_2(\mathfrak{X}^2,F^2)$ such that the weak limit of the corresponding $U$-statistics does not exist.
\end{proposition}

\begin{remark}                         %��������� 4.
In~\cite{Bystr:BORVOL}, in the case of dependent observations, another approach was proposed for description of the limit distribution of  canonical Von Mises statistics as a multiple stochastic integral of the kernel under consideration,  with respect to increments  of a centered Gaussian process with a covariance function defined by
joint distributions of the random variables~$\{X_i\}$.
In the i.i.d. case, such dual description of the limit low is well known (for example, see~\cite{KB:BORVOL}).
\end{remark}

However, the stationary sequence~$\{X_i\}$ in~\cite{Bystr:BORVOL}
must satisfy a stronger $\psi$-mixing condition.
In the same time, in contrast to the present paper,
condition~(\ref{abs:BORVOL}) and the regularity conditions for the kernel and the basis functions of  Theorem~2 were replaced in~\cite{Bystr:BORVOL} with the more natural condition of integrability of the kernel on all the diagonal subspaces.
Note that the above-mentioned regularity condition and ~(\ref{abs:BORVOL})
imply the boundedness of the kernel under consideration, i.~e.,
 the above-mentioned condition of the kernel integrability on the diagonal subspaces in~\cite{Bystr:BORVOL} is fulfilled.
  % equivalent to summability of the coefficients $f_{i_1,\dots,i_m}$ on all %the diagonal subspaces of the region of summation for the multiindex
 %$(i_1,\dots,i_m)$.

It is not clear importance of restriction~(\ref{abs:BORVOL})
outside the diagonal subspaces (i.~e.,~on the set of all pairwise distinct subscripts) to approximate~$U$- and $V$-statistics of an arbitrary order for dependent trials.

Mention also the important particular case when the bivariate kernel
of a $V$-statistic is represented as an inner product $f(x,y)=(x,y)$ in
a separable Hilbert space. In this case, the corresponding Von Mises statistic coincides with the Euclidian norm squared of a normalized sum of weakly dependent centered observations and we deal with the Central Limit Theorem (with respect to the class of all centered balls) for Hilbert-space-valued weakly dependent observation which was proved under various mixing conditions and the existence of the moment~${\mathbb E}(X_1,X_1)$ (for example, see~\cite{tikh:BORVOL}).

\section{Proof of Theorems~1 and~2 and Propositions~2 and~4}   %�������� 3

First of all, we formulate the following two auxiliary assertions which are
versions of the classical Rosenthal moment inequality for sums of independent random variables.

\begin{theorem}[A {\rm{{\cite{utev:BORVOL}}}}]                  %{������� �}
Let $\xi_i$~be a sequence of centered random variables having finite moments of order~$t\geq 2$ and satisfying $\varphi$-mixing with the restriction
$\varphi:=\sum_{k=1}^{\infty}\varphi^{1/2}(2^k)<\infty$. Then, for $t\geq 2$,
the following inequality holds:
$$
\me\max_{1\leq k\leq n}|S_k|^t
\leq\big(tc(\varphi)
    \big)^t
    \left(\sum_{i=1}^{n}\me |\xi_i|^t+
          \Bigg(\sum_{i=1}^{n}\me |\xi_i|^2
          \Bigg)^{t/2}
    \right),
$$
where the constant $c(\varphi)$ depends on~$\varphi$  only.
\end{theorem}

\begin{theorem}[� {\rm{{\cite{doukhan:BORVOL}}}}]                    %{������� �}
Let $\xi_i$~be a sequence of centered random variables with finite moments of order~$t\geq 2$ satisfying $\alpha$-mixing and let,
for some $\varepsilon>0$ and even~$c\geq t$, the following condition be fulfilled:
$$
\sum_{k=1}^{\infty}k^{c-2}\alpha^{\varepsilon/(c+\varepsilon)}(k)<\infty.
$$
Then there exists a constant~$C$ depending on~$t$ and the mixing coefficient~$\alpha(k)$, such that
$$
\me |S_n|^t\leq C
\max
\left(\sum_{k=1}^{n}
\big(\me |\xi_k|^{t+\varepsilon}
\big)^{t/(t+\varepsilon)},
     \Bigg(\sum_{k=1}^n\big(\me |\xi_k|^{2+\varepsilon}
                       \big)^{2/(2+\varepsilon)}
     \Bigg)^{t/2}
\right).
$$
\end{theorem}

\begin{proof}[Proof of Theorem\/~{\rm 1}]
Consider the following partial sum for the above-mentioned multiple series expansion of the kernel~$f(t_1,\dots,t_m)$:
$$
f_N(t_1,\dots,t_m):=
\sum_{1\leq i_1,\dots,i_m\leq N}
f_{i_1,\dots,i_m}e_{i_1}(t_1)\cdots\,e_{i_m}(t_m).
$$
It is clear that, due to linearity of the functional~$U_n(\boldsymbol{\cdot})$,
we have
$$
U_n(f_N)=\sum_{1\leq i_1,\dots,i_m\leq N}f_{i_1,\dots,i_m}
U_n(e_{i_1}\cdots\,e_{i_m}),
$$
where
$$
U_n(e_{i_1}\cdots\,e_{i_m}):=n^{-m/2}
\sum_{1\leq j_1\neq\cdots\,\neq j_m\leq n}
e_{i_1}(X_{j_1})\cdots\,e_{i_m}(X_{j_m}).
$$
Further, by analogy with the above-mentioned arguments for independent observations, we represent the
$U$-statistic as a sum of Von Mises statistics in which of them the summation is taken over all the subscripts (not necessarily pairwise distinct).
Further, changing the order of summation, we start to study the random variables
$$
\frac{1}{\sqrt{n}}
\sum_{j=1}^ne_i(X_j),
\ \frac{1}{n}\sum_{j=1}^ne_{i_1}(X_j)e_{i_2}(X_j),\,\dots,\,
\frac{1}{n^{k/2}}
\sum_{j=1}^ne_{i_1}(X_j)\cdots\,e_{i_k}(X_j).
$$
For each~$k>2$, these sums converge to zero in probability  as $n\rightarrow\infty$ since, by condition~(\ref{fibase:BORVOL}), there exists the finite moment $\me\big\vert e_{i_1}(X_j)\cdots\,e_{i_k}(X_j)\big\vert$.
Hence, by the law of large numbers for weakly dependent random variables,
the following convergence is valid:
$$
\frac1n
\sum_{j=1}^ne_{i_1}(X_j)
\cdots\,e_{i_k}(X_j)\to\me e_{i_1}(X_j)
\cdots\,e_{i_k}(X_j)\ \ \mbox{in probability}.
$$
So, for $k>2$, we obtain
$$
\frac1{n^{k/2}}
\sum_{j=1}^ne_{i_1}(X_j)
\cdots\,e_{i_k}(X_j)\to 0
\ \ \mbox{in probability}.
$$
Therefore, the summands containing such sums as factors, converge to zero in probability  as well.
For~$k=2$, we also apply the law of large numbers to the sums mentioned. Due to orthonormality of the basis, this limiting values coincide with the Kronecker symbol~$\delta_{i,k}$.
Therefore, the limiting form of the partial sum is quite similar to that in the i.i.d. case:
$$
U_n(f_N)\stackrel{d}\rightarrow
\sum_{1\leq i_1,\dots,i_m\leq N}f_{i_1,\dots,i_m}
\prod_{j=1}^{\infty}H_{\nu_j(i_1,\dots,i_m)}(\tau_j)=:\eta_{\infty N},
$$
where the random variables~$\{\tau_k\}$ have joint Gaussian distributions with the covariance matrix defined in~(\ref{cov:BORVOL}). The last assertion is a consequence of the multivariate Central Limit Theorem for finite-dimensional projections of empirical processes based on stationary observations and indexed by the family of functions~$\{e_k\}$.
In turn, the last result follows from the univariate CLT for stationary sequences and from the Cram{\'e}r -- Wold method to reduce the multivariate case to the univariate one
\cite[Theorem~20.1]{Billingsley:BORVOL}.

Further, using the well-known covariance inequalities for stationary sequences, we prove  finiteness of the covariance we need.

%\newpage
(a) For $\varphi$-mixing, from \cite[�������~17.2.3]{Ibragim:BORVOL} and
condition~(\ref{fi:BORVOL}) it follows that
%\vskip-7pt\noindent
%$$
\begin{align*}
|\me\tau_k\tau_l|
&\leq\big\vert\me e_k(X_1)e_l(X_1)
     \big\vert+
\sum_{j=1}^{\infty}
\Big[\big\vert\me e_k(X_1)e_l(X_{j+1})
     \big\vert+
     \big\vert\me e_l(X_1)e_k(X_{j+1})
     \big\vert
\Big]\\
&\leq\delta_{k,l}+2\sum_{j=1}^{\infty}\varphi^{1/2}(j)<\infty.
\end{align*}
%$$

(b) For $\alpha$-mixing, due to \cite[�������~17.2.2]{Ibragim:BORVOL} and
condition~(\ref{alpha:BORVOL}), we have
%\vskip-7pt\noindent
%$$
\begin{align*}
|\me\tau_k\tau_l|
&\leq\big\vert\me e_k(X_1)e_l(X_1)
     \big\vert+
\sum_{j=1}^{\infty}
\Big[\big\vert\me e_k(X_1)e_l(X_{j+1})
     \big\vert+
     \big\vert\me e_l(X_1)e_k(X_{j+1})
     \big\vert
\Big]\\
&\leq\delta_{k,l}+C
\sum_{j=1}^{\infty}\alpha^{\varepsilon/(2+\varepsilon)}(j)<\infty.
\end{align*}
%$$
Notice that, in cases~(a) and~(b), the covariances~$\me\tau_k\tau_l$
{\it uniformly bounded}.

We now represent the prelimit and limit statistics in such a way:
%$$
\begin{gather*}
U_n(f)=U_n(f_N)+U_n(f-f_N),\\
\noalign{\vskip7pt}
\eta_{\infty}=\eta_{\infty N}+\overline{\eta_{\infty N}},
\end{gather*}
%$$
%\vskip-7pt\noindent
where
%\vskip-7pt\noindent
$$
\overline{\eta_{\infty N}}:=
\sum_{\max\{i_k\}>N}f_{i_1,\dots,i_m}
\prod_{j=1}^{\infty}H_{\nu_j(i_1,\dots,i_m)}(\tau_j).
$$

Prove that,  as $N\rightarrow\infty$, the random variables~$U_n(f-f_N)$ and~$\overline{\eta_{\infty N}}$
converge to zero in mean. It is clear that
$$
\me\big\vert\overline{\eta_{\infty N}}
          \big\vert\leq
\sum_{\max\{i_k\}>N}|f_{i_1,\dots,i_m}|\cdot\me
\Bigg\vert
\prod_{j=1}^{\infty}H_{\nu_j(i_1,\dots,i_m)}(\tau_j)
\Bigg\vert,
$$
where, as was noted before,
$\prod_{j=1}^{\infty}H_{\nu_j(i_1,\dots,i_m)}(\tau_j)$~is a polynomial of order~$m$ of the variables
 $\tau_{j_1},\dots,\tau_{j_s}$ \big(see the notation
in~(\ref{represent:BORVOL})\big), and all the coefficients of the polynomial depending on the multiplicity~$r_l$ of subscripts in a fixed collection
$(i_1,\dots,i_m)$, have a universal upper bound depending only on~$m$.
The absolute moment of each monomial of this polynomial is bounded uniformly in~$i_1,\dots,i_m$. It suffices to show this property for the high-order member of the polynomial. Instead, by H{\"o}lder's inequality and relation $\sum_{l\le s}r_l=m$, we have
$$
\me\big\vert\tau^{r_1}_{j_1}\cdots\,\tau^{r_s}_{j_s}
          \big\vert
\leq\big(\me |\tau_{j_1}|^m
    \big)^{r_1/m}\cdots\,
    \big(\me |\tau_{j_s}|^m
    \big)^{r_s/m}.
$$
Since the variances of the limit Gaussian random variables~$\tau_i$ are uniformly bounded, their absolute moments of order~$m$ are uniformly bounded as well. Hence, the moment
$\me\big\vert\prod_{j=1}^{\infty}H_{\nu_j(i_1,\dots,i_m)}(\tau_j)
\big\vert$
is bounded {\it uniformly over all\/}  $i_1,\dots,i_m$. Therefore, condition~(\ref{abs:BORVOL}) implies the limit relation
$\me\big\vert\overline{\eta_{\infty N}}\big\vert\rightarrow 0$ as
$N\rightarrow\infty$.

Consider now the series tail in the expansion of $U$-statistic.  We have
%$$
\begin{align*}
\me\big\vert U_n(f-f_N)
          \big\vert
&=n^{-m/2}\me
\Bigg\vert\sum_{1\leq j_1\neq\cdots\,\neq j_m\leq n}
     \big(f(X_{j_1},\dots,X_{j_m})-f_N(X_{j_1},\dots,X_{j_m})
     \big)
\Bigg\vert\\
&=n^{-m/2}\me
\Bigg\vert\sum_{1\leq j_1\neq\cdots\,\neq j_m\leq n}
          \sum_{\max\{i_k\}>N}f_{i_1,\dots,i_m}
                             e_{i_1}(X_{j_1})\cdots\,e_{i_m}(X_{j_m})
\Bigg\vert\\
&=n^{-m/2}\me
\Bigg\vert\sum_{\max\{i_k\}>N}
          \sum_{1\leq j_1\neq\cdots\,\neq j_m\leq n}
          f_{i_1,\dots,i_m}
          e_{i_1}(X_{j_1})\cdots\,e_{i_m}(X_{j_m})
\Bigg\vert\\
&\leq n^{-m/2}\sum_{\max\{i_k\}>N}|f_{i_1,\dots,i_m}|\me
\Bigg\vert\sum_{1\leq j_1\neq\cdots\,\neq j_m\leq n}\!\!\!\!
          e_{i_1}(X_{j_1})\cdots\,e_{i_m}(X_{j_m})
\Bigg\vert\\
&=\sum_{\max\{i_k\}>N}|f_{i_1,\dots,i_m}|\me
\Bigg\vert n^{-m/2}\!\!\!\!\sum_{1\leq j_1\neq\cdots\,\neq j_m\leq n}
  e_{i_1}(X_{j_1})\cdots\,e_{i_m}(X_{j_m})
\Bigg\vert.
\end{align*}
%$$
Consider the value
%$$
\begin{equation}\label{dopred:BORVOL}          %(22)
n^{-m/2}
\sum_{1\leq j_1\neq\cdots\,\neq j_m\leq n}
e_{i_1}(X_{j_1})\cdots\,e_{i_m}(X_{j_m}).
\end{equation}
%$$
Due to summability of the coefficients~$|f_{i_1,\dots,i_m}|$, it suffices to prove the uniform boundedness over all $i_1,\dots,i_m$ of the value in~(\ref{dopred:BORVOL}). As before, we add and subtract into the multiple sum in ~(\ref{dopred:BORVOL}) summations over diagonal subspaces obtaining a linear combination of monomials of the form
$$
\frac{\sum_{j=1}^ne_{l_1}(X_j)\cdots\,e_{l_{m_1}}(X_j)}
     {n^{m_1/2}}
\times\dots\times
\frac{\sum_{j=1}^ne_{l_{m_1+\cdots\,+m_{k-1}+1}}(X_j)\cdots\,e_{l_{m}}(X_j)}
     {n^{m_k/2}},
$$
where $(l_1,\dots,l_m)$~is a permutation of the multiindex
$(i_1,\dots,i_m)$ and the natural number~$m_i$ satisfies  the relation
$m_1+\cdots\,+m_k=m$, $k\le m$, and the number of  such monomials is bounded by a number depending on~$m$ only. Thus, proving the uniform boundedness
of  the absolute moment of such monomials we now prove the uniform boundedness of the absolute moment of the value in~(\ref{dopred:BORVOL}). We have
%$$
\begin{align*}
&\me
\Bigg\vert
      \frac{\sum_{j=1}^ne_{l_1}(X_j)\cdots\,e_{l_{m_1}}(X_j)}
           {n^{m_1/2}}
\times\dots\times
      \frac{\sum_{j=1}^ne_{l_{m_1+\cdots+m_{k-1}+1}}(X_j)\cdots\,e_{l_m}(X_j)}
           {n^{m_k/2}}
\Bigg\vert\\
&\qquad\leq\prod_{r=1}^{k}
\left(\me
      \Bigg\vert\frac{\sum_{j=1}^ne_{l_{m_1+\cdots+m_{r-1}+1}}
                     (X_j)\cdots\,e_{l_{m_1+\cdots+m_r}}(X_j)}
                     {n^{m_r/2}}
      \Bigg\vert^{m/m_r}
\right)^{m_r/m},
\end{align*}
%$$
where $m_0=0$.
Prove that each of the factors above is uniformly bounded. Without loss of generality, one can verify this only for the first factor. Since in general the summands $e_{l_1}(X_j)\cdots\,e_{l_{m_1}}(X_j)$ are not centered, we consider  separately the following two cases.

First, let $m_1=1$. In this case, we estimate moments of a sum of centered weakly dependent random variables. Under $\varphi$-mixing, we use Theorem~A. Then
%$$
\begin{align*}
\me
\Bigg\vert\frac1{\sqrt{n}}
          \sum_{j=1}^ne_i(X_{j})
\Bigg\vert^{m}
&\!\!\!\leq C_0(m)
\Bigg(\bigg(\frac1n
            \sum_{j=1}^n\me
            \big(e_i(X_j)
            \big)^2
      \bigg)^{m/2}\hskip-2.8mm+n^{-m/2}
      \sum_{j=1}^{n}\me
      \big\vert e_i(X_{j})
      \big\vert^m
\Bigg)\\
&\leq C_0(m)\big(1+n^{-(m/2+1)}C(m)
            \big)\\
&\leq C_0(m)\big(1+C(m)
             \big).
\end{align*}
%$$
Under $\alpha$-mixing, by Theorem~B, we obtain
%$$
\begin{align*}
\me
\Bigg\vert\frac{1}{\sqrt{n}}
          \sum_{j=1}^{n}e_{i}(X_{j})
\Bigg\vert^{m}
&\leq C_0(m)
\Bigg(\frac{1}{n}
      \sum_{j=1}^n
      \Big(\me\big\vert e_i(X_j)
                     \big\vert^{2+\varepsilon}
      \Big)^{2/(2+\varepsilon)}
\Bigg)^{m/2}\\
&\qquad+C_0(m)n^{-m/2}
\sum_{j=1}^n
\Big(\me\big\vert e_i(X_j)
               \big\vert^{m+\varepsilon}
\Big)^{m/(m+\varepsilon)}\\
&\leq C_1(m).
\end{align*}
%$$
%\goodbreak

Now, let $m_1\geq 2$. Then, by the arithmetical convexity inequality
for sums and by H{\"o}lder's  inequality, we obtain the following simple estimate:
%$$
\begin{align*}
\me
\Bigg\vert\frac{\sum_{j=1}^ne_{l_1}(X_j)\cdots\,e_{l_{m_1}}(X_j)}
               {n^{m_1/2}}
\Bigg\vert^{m/m_1}
\leq n^{m/m_1}n^{-m/2}\me
\big\vert e_{l_1}(X_1)\cdots\,e_{l_{m_1}}(X_1)
\big\vert^{m/m_1}
\le\sup_i\me\big\vert e_{i}(X_1)
                    \big\vert^m.
\end{align*}
%$$
Thus,
$$
\me\big\vert U_n(f-f_N)
          \big\vert
\leq C\sum_{\max\{i_k\}>N}
\big\vert f_{i_1,\dots,i_m}
\big\vert\rightarrow 0\quad\mbox{as}\quad N\to\infty.
$$
Theorem~1 is proved.
\end{proof}

\begin{proof}[Proof of Theorem\/~{\rm 2}]
Without loss of generality, we assume that the separable metric space~$\mathfrak X$ is the support of the distribution~$F$. It means that the space~$\mathfrak X$ does not contain open balls having zero
$F$-measure. Since all the basis functions~$e_k(t)$ in~(\ref{kernel:BORVOL})
are continuous and bounded uniformly in~$t$ and~$k$,
 by the Lebesgue theorem
on dominated convergence, these  above facts imply continuity of the series in~(\ref{kernel:BORVOL}) under condition~(\ref{abs:BORVOL}). It is easy to see that, in this case, the equality in~(\ref{kernel:BORVOL}) turns into the identity of the variables
$t_1,\dots,t_m$ since equality of two functions on an everywhere dense subset of~$\mathfrak X$ implies their coincidence everywhere in~$\mathfrak X$.
Therefore, we can replace the variables $t_1,\dots,t_m$ in identity~(\ref{kernel:BORVOL}) with {\it arbitrarily dependent\/} observations,
in particular, coincident ones. This is a principal distinction from the proof of the previous theorem. Hence, for the {\it all elementary events\/}, the following representation holds:
$$
V_n(f)=\sum_{i_1,\dots,i_m=1}^{\infty}f_{i_1,\dots,i_m}n^{-1/2}
       \sum_{j=1}^ne_{i_1}(X_j)\cdots\,n^{-1/2}
       \sum_{j=1}^ne_{i_m}(X_j).
$$
Further arguments actually repeat the arguments
(even in a simpler version) of the proof of Theorem~1.
Theorem~2 is proved.
\end{proof}

\begin{proof}[Proof of Proposition\/~{\rm 2}]
Let $\{Y_i;\,i\ge 1\}$~be a sequence of independent random variables uniformly distributed on~$[-1,1]$, and let
$\{\xi_i;\,i\ge 1\}$~be a sequence of independent symmetric Bernoulli random variables
which are independent of~$\{Y_i\}$ as well. Set $X_i=Y_{i+\xi_i}$. The random variables~$\{X_i\}$
form a stationary $1$-dependent sequence. Notice that, in this case, the random variables~$X_i$ are uniformly distributed on~$[-1,1]$ as well. Thus, the stationary sequence~$\{X_i\}$ satisfies $\varphi$-mixing condition and the restriction in~(\ref{fi:BORVOL}). It is clear that, for the corresponding independent copies~$X_1^*$ and~$X_2^*$, we have
$P(X_1^*=X_2^*)=0$, but for the originals, we obtain
%\vskip-1pt\noindent
$$
P(X_1=X_2)=P(\xi_1=1)P(\xi_2=0)=1/4.
$$
Notice that, in the example under consideration, the basis functions and the coefficients in the series expansion in~(\ref{kernel:BORVOL}) do not depend on the values of the kernel~$f$ on the diagonal due to continuity of the marginal distribution. Let
%\vskip-1pt\noindent
$$
\sup_{|t_1|,|t_2|\le 1}\big\vert f(t_1,t_2)
                       \big\vert\le 1.
$$
Now, change the diagonal values setting $f(t,t)\equiv 1+\beta$,
for all $t\in[-1,1]$, where~$\beta>0$. Then
$P\big(f(X_1,X_2)=1+\beta\big)=1/4$. In the same time, the series on the right-hand side of~(\ref{kernel:BORVOL}) does not depend on~$\beta$.

It is easy to show that, in the example under consideration, the limit law  will essentially differs from~(\ref{eagleresult:BORVOL}).
Instead, represent the bivariate  $U$-statistic with the symmetric canonical kernel~$f(t_1,t_2)$ in the form
%\vskip-7pt\noindent
%$$
\begin{equation} \label{U:BORVOL}                      %(23)
U_n=\frac2n
    \sum_{i<n}f(X_i,X_{i+1})+{\tu}_n.
\end{equation}
%$$

Limit behavior of the first sum on the right-hand side~(\ref{U:BORVOL})
is simply studied since the random variables
$\big\{f(X_i,X_{i+1});\,i\ge 1\big\}$ form $2$-dependent
stationary sequence for which the strong law of large numbers is valid.  Hence, as $n\to\infty$,
the limit relation
%$$
\begin{equation}\label{diagonal:BORVOL}              %(24)
\frac2n
\sum_{i<n}f(X_i,X_{i+1})\to 2{\mathbb E}f(X_1,X_{2})=
{\mathbb E}f(Y_1,Y_{1})/2
\end{equation}
%$$
holds with probability~$1$. Under the above-mentioned restrictions on the kernel, the limit equals~$\frac{1+\beta}{2}$.

The statistic~${\tu}_n$ is based on the random variables~$f(X_i,X_j)$ under the condition $|i-j|\ge 2$ which provides independence of the random variables~$X_i$ and~$X_j$. Since the distribution of~${\tu}_n$ (hence, and the limit one if it exists)  does not depend on~$\beta$,
this proves the assertion.

We now compute the above-mentioned limit law  for the sequence~$\{{\tu}_n\}$ under absolutely summability of the sequence $\{\lambda_k\}$ in the series expansion in~(\ref{eagl:BORVOL}) for the kernel~$f(t_1,t_2)$ under consideration.
Notice that the kernel with the indicated properties exists (see the example
after Remark~2 in the previous section).
We note once again that, in the example under consideration, all the eigenvalues~$\lambda_k$ and all the basis functions~$e_k(t)$ do not depend on the diagonal values of the kernel (i.~e., do not depend on~$\beta$
under the additional above-mentioned restrictions on the kernel). Since the statistic~${\tu}_n$ consists only of the random values~$f(X_i,X_j)$ under the condition~$|i-j|\ge 2$ which provides independence of the random variables~$X_i$ and~$X_j$, we can replace~$f(X_i,X_j)$ with the corresponding double series from~(\ref{eagl:BORVOL}). Then
%\vskip-1pt\noindent
%$$
\begin{align}\label{limit:BORVOL}                      %(25)
{\tu}_n
&=\frac1n
  \sum_{i,j\le n:\,|i-j|\ge 2}\,
  \sum_{k\ge 1}\lambda_ke_k(X_i)e_k(X_j)\nonumber\\
&=\sum_{k\ge 1}\lambda_k
\left\{\bigg(\frac{1}{\sqrt n}
             \sum_{i\le n}e_k(X_i)
       \bigg)^2-\frac2n
                \sum_{i<n}e_k(X_i)e_k(X_{i+1})-
                \frac1n
                \sum_{i\le n}e^2_k(X_i)
\right\}\nonumber\\
&\qquad\stackrel{d}\rightarrow\sum_{k\ge 1}\lambda_k(\tau^2_k-3/2),
\end{align}
%$$
%\vskip-7pt\noindent
where $\{\tau_k\}$~is a centered Gaussian sequence with covariances~(\ref{cov:BORVOL}). To compute the weak limit as $n\to\infty$,
we use, first, the fact of the weak convergence of the collection of sequences
(in other words, their finite-dimensional distributions)
$\big\{\frac1{\sqrt n}\sum_{i\le n}e_k(X_i);\,k\ge 1\big\}$ to the Gaussian limit~$\{\tau_k;\,k\ge 1\}$, second, the laws of large numbers both for the $1$-dependent and
$2$-dependent stationary sequences $\{e^2_k(X_i);\,i\ge 1\}$ and
$\big\{e_k(X_i)e_k(X_{i+1});\,i\ge 1\big\}$ respectively, taking into account Proposition~3  as well as  the relations  ${\mathbb E}e^2_k(X_1)=1$ and
${\mathbb E}e_k(X_1)e_k(X_2)=1/4$ for all~$k\ge 1$.
Moreover, we used simple upper bounds for the absolute moments of remainders of the series in~(\ref{limit:BORVOL}). So, from~(\hbox{\ref{U:BORVOL})--(\ref{limit:BORVOL}}) we obtain the form of the limit law:
$$
U_n\stackrel{d}\rightarrow{\mathbb E}f(Y_1,Y_{1})/2+
\sum_{k\ge 1}\lambda_k
\left(\tau^2_k-3/2
\right).
$$
%\vskip-7pt\noindent
The right-hand side coincides with~(\ref{eagleresult:BORVOL}), say, under the conditions of Remark~2
(or of Theorem~2) but does not coincide under the above-mentioned restrictions on diagonal values of the kernel~$f$.
Proposition~2 is proved.
\end{proof}

\begin{proof}[Proof of Proposition\/~{\rm 4}]
Define the $1$-dependent stationary sequence~$\{X_i\}$ by the scheme of Proposition~2 using i.i.d. random variables~$\{Y_i\}$ with the following distribution (different from that in Proposition~2):
%$$
\begin{equation*}
Y_1=
\begin{cases}
\hfill k&\mbox{with probability}\,\, p_k=2^{-k-1},\\
      -k&\mbox{with probability}\,\, p_k=2^{-k-1}
\end{cases}
\end{equation*}
%$$
for all natural~$k$. We assume that $\mathfrak{X}$ coincides with the support of the distribution~$F$, i.~e., with the set of all nonzero integers. It is easy to verify that the stationary sequence~$\{X_i\}$ satisfies condition~$\AC$ since, under an arbitrary dependence of the components
of the vector $(X_{i_1},\dots,X_{i_m})$, the support of its distribution is contained in the support of the $m$-variate distribution of $(X^*_1,\dots,X^*_m)$.

Consider in $L_2(\mathfrak{X},F)$ the sequence of functions
%$$
\begin{gather*}
e_0\equiv 1;\\
\noalign{\vskip3pt}
e_k(t)=\begin{cases}
\hfill{2^{k/2}} &\mbox{if}\ t=k;\\
    - {2^{k/2}} &\mbox{if}\ t=-k;\\
\hfill      0     &\mbox{otherwise}.
       \end{cases}
\end{gather*}
%$$
It is clear that this sequence is an orthonormal basis in the space~$L_2(\mathfrak{X},F)$. Let
$\{\lambda_k\}_{k\geq 1}$~be an arbitrary sequence of positive numbers summable squared. Define a symmetric canonical kernel~$f(t,s)$
by the formula
$$
f(t,s):=\sum_{k=1}^{\infty}\lambda_ke_k(t)e_k(s).
$$
%\vskip-7pt\noindent
It is clear that
%\vskip-7pt\noindent
%$$
\begin{equation*}
f(k,l)=
\begin{cases}
\hfill\lambda_{|k|}2^{|k|} &\mbox{if}\ k=l;\\
     -\lambda_{|k|}2^{|k|} &\mbox{if}\ k=-l;\\
\hfill             0       &\mbox{otherwise.}
\end{cases}
\end{equation*}
%$$
Moreover, $\{\lambda_k\}$ and~$\big\{e_k(t)\big\}$ are the respective eigenvalues and eigenfunctions of the integral operator with the kernel~$f(t,s)$.

We now employ the expansion in~(23).
Using the same arguments as in~(25),
we obtain
%\vskip-1pt\noindent
%$$
\begin{equation}\label{series:BORVOL}            %(26)
{\tu}_n\stackrel{d}\rightarrow\sum_{k=1}^{\infty}\lambda_k(\tau_k^2-3/2),
\end{equation}
%$$
where $\{\tau_k\}$~is a Gaussian sequence with covariances~(\ref{cov:BORVOL})
which, in the case under consideration, are calculated in such a way:
$\me\tau_k\tau_l=0$ if~$k\neq l$,
and $\me\tau_k^2=3/2$, i.~e., the limit sequence consists of independent centered Gaussian random variables with the common variance~$3/2$. Therefore, by the Kolmogorov Three Series theorem (for example, see~\cite{GS:BORVOL}),
the series on the right-hand side of (\ref{series:BORVOL})
converges almost surely; we need  only summability squared
of the sequence $\{\lambda_k\}$ at that.
Consider in detail the limit behavior of the first sum in expansion~(23).
To transform this sum, use the following simple identity:
%$$
\begin{multline*}
%\begin{align*}
f(X_i,X_{i+1})=\xi_i(1-\xi_{i+1})f(Y_{i+1},Y_{i+1})
+\xi_{i}\xi_{i+1}f(Y_{i+1},Y_{i+2})\\
+(1-\xi_i)\xi_{i+1}f(Y_i,Y_{i+2})+
(1-\xi_i)(1-\xi_{i+1})f(Y_{i},Y_{i+1}).
\end{multline*}
%$$
\goodbreak

\noindent Every term of the fourth summands on the right-hand side of this identity is the
$i$th member of the corresponding  $1$- or $2$-dependent stationary
sequence.
Hence, for each partial sum consisting of these summands except the first one, the strong law of large numbers is applicable and the corresponding normalized sums tend to zero almost surely due to degeneracy of the kernel.
However, the law of large numbers is not applicable to the sum consisting of the summands
$\nu_i:=\xi_i(1-\xi_{i+1})f(Y_{i+1},Y_{i+1})$
since the expectation $\mathbb{E}\nu_1=\frac{1}{4}\sum_k\lambda_k$
may be infinite.
In this case, the limit behavior of the first sum on the right-hand side of  (23) is described by the following

\begin{lemma} Let $\hskip-1pt \{\mu_k\}$ be a sequence of i.i.d. nonnegative  random variables with infinite expectations. Then
$$
\lim_{n\to\infty}\frac{1}{n}\sum_{k\le n}\mu_k=\infty.
$$
almost surely.
\end{lemma}

The {\it proof\/} of this simple assertion is a direct consequence of the Borel~--- Kantelli lemma since the condition
$\mathbb{E}\mu_1=\infty$ is equivalent to divergence of the series
$\sum_k\mathbb{P}(\mu_k>Nk)
$
for an arbitrary fixed $N>0.$ From here we immediately conclude that
$$\liminf_n\frac{1}{n}\sum_{k\le n}\mu_k>N$$
almost surely which was to be proved.
%\vskip7pt

Notice that the requirement of nonnegativity of the summands is essential. The corresponding counterexample is constructed by i.i.d. random variables $\mu_k$ with the Cauchy distribution.

Since the nonnegative identically distributed random variables
$\{\nu_k\}$ form a $1$-dependent stationary sequence, splitting the corresponding partial sum into two ones (each of them consists of independent summands),
we easily reduce the problem to that  described above.

Thus, in the case under consideration,
the existence of a finite weak limit for the $U$-statistics
is equivalent to summability of the coefficients $\{\lambda_k\}$.
Proposition~4 is proved.
\end{proof}

{\vskip3.1ex plus 1ex minus.2ex
\begin{center}ACKNOWLEDGMENTS
\end{center}
\par\vskip1pt}

This research was
supported by the Russian Foundation for Basic Research
(grant~06--01--00738) and by the Ministry of Higher Education and Science
of the Russian Federation
(grant~���.2.1.1.1379).

\end{document}